\documentclass[11pt]{amsart}
\usepackage{graphicx,epsfig}
\usepackage{cite}
\usepackage{mathtools}
\usepackage{enumerate}
\usepackage{color}

 \usepackage[bookmarks=true,
bookmarksnumbered=true, breaklinks=true,
pdfstartview=FitH, hyperfigures=false,
plainpages=false, naturalnames=true,
colorlinks=false,
pdfpagelabels]{hyperref}

\makeatletter
\newtheorem*{rep@theorem}{\rep@title}
\newcommand{\newreptheorem}[2]{%
\newenvironment{rep#1}[1]{%
 \def\rep@title{#2 \ref{##1}}%
 \begin{rep@theorem}}%
 {\end{rep@theorem}}}
\makeatother
\newreptheorem{theorem}{Theorem}
\newreptheorem{corollary}{Corollary}

\newcommand{\R}{\mathbb{R}} 
\def \Rn{{\R^n}}
\newcommand{\hid}{m}
\def \Rnhi{{\R^{n\hid}}}

\usepackage{comment}
\def \B{B_2^n}
\newcommand{\N}{\mathbb{N}}
\def \E{\mathbb E}

\def \s{\mathbb{S}^{n-1}}
\def \S{\mathbb{S}^{n\hid-1}}

\newcommand{\vol}{\text{\rm Vol}}

\def \pp{\Pi^{\circ}}

\def \cov{g_{K}}

\def \conv{\operatorname{conv}}

\def \PP{\Pi^{\circ,\hid}}

\def \Cov{g_{K,\hid}}
\mathtoolsset{showonlyrefs}

\newtheorem{theorem}{Theorem}[section]
\newtheorem{definition}[theorem]{Definition}
\newtheorem{lemma}[theorem]{Lemma}
\newtheorem{proposition}[theorem]{Proposition}
\newtheorem{corollary}[theorem]{Corollary}

\title[Higher Order Projection Body]{Some comments on the \\ mth-order Projection Bodies}

\author[D. LANGHARST]{Dylan Langharst}
\address{Carnegie Mellon University\\ Department of Mathematical Sciences \\ Pittsburgh, PA 15213, USA\\ OrcID: 0000-0002-4767-3371}
\email{dlanghar@andrew.cmu.edu}


\begin{document}

\begin{abstract}
The celebrated Petty's projection inequality is a sharp upper bound for the volume of the polar projection body of a convex body. Lutwak introduced the concept of mixed projection bodies and extended Petty's projection inequality. Alonso-Guti\'{e}rrez later did a stability result for Petty's projection inequality.

In 1970, Schneider introduced the $m$th-order setting and extended the difference body to that setting. In a previous work, we, working with Haddad, Putterman, Roysdon, and Ye, established an extension of the projection body operator to this setting. In this note, we continue this study for the mixed projection body operator as well as the question of stability.
\end{abstract}

\maketitle

\vspace{-0.75cm}

\maketitle

\section{Projection Bodies}
One of the earliest inequalities one encounters in the study of geometric shapes is the isoperimetric inequality: let $A\subset\R^n$ be a set of finite perimeter and volume. Then,
\begin{equation}
\label{eq:iso}
    \vol_{n-1}(\partial A)\geq n\omega_n^\frac{1}{n}\vol_n(A)^\frac{n-1}{n},
\end{equation}
with equality if and only if $A$ is a Euclidean ball up to null sets. Some definitions are in order; here, $\R^n$ is the $n$-dimensional Euclidean space, $\B$ is the Euclidean unit ball with volume (Lebesgue measure) $\omega_n$ and boundary $\s=\partial \B$ the unit sphere, $\vol_n(A)$ is the $n$-dimensional Lebesgue measure of $A$, and $\vol_{n-1}(\partial A)$ is the surface area of $A$. For now, we do not define surface area so precisely. 

A flagrant handicap of \eqref{eq:iso} is the fact that it is not affine invariant: If $A$ is replaced by $TA$, where $T$ is a volume-preserving affine transformation, then the left-hand side of \eqref{eq:iso} may change, but the right-hand side will not. A natural question is then, therefore: does there exist an affine invariant isoperimetric inequality? In the course of this article, we hope to illustrate that the setting of \textit{convex} geometry is a natural path forward when dealing with these types of questions. We recall that $K\subset\R^n$ is a convex body if it is convex, compact, and has non-empty interior. We can now give a more precise definition of surface area: 
\begin{equation}
\label{eq:surface_area}
    \vol_{n-1}(\partial K)=\lim_{\epsilon \to 0}\frac{\vol_n(K+\epsilon \B)-\vol_n(K)}{\epsilon}.
\end{equation}
Here, $A+B=\{a+b:a\in A,b\in B\}$ is the Minkowski sum of two Borel sets. We will also consider in certain instances compact, convex sets that may have empty interior; implicitly, these are assumed to not be the empty set.

A natural solution to an affine invariant isoperimetric inequality for convex bodies is {\bf Petty's projection inequality:} for a convex body $K$, one has
\begin{equation}
    \label{eq:petty_projection}
    \vol_n(K)^{n-1}\vol_n(\Pi^\circ K)\leq \left(\frac{\omega_n}{\omega_{n-1}}\right)^n,
\end{equation}
with equality if and only if $K$ is an ellipsoid. This was shown by Petty in 1971 \cite{CMP71}. The body $\pp K$ is the polar projection body of $K$, which is an origin-symmetric convex body derived from $K$. Furthermore, as we will see from its definition, the left-hand side of \eqref{eq:petty_projection} is affine invariant. Before we delve into that, we first elucidate how \eqref{eq:petty_projection} implies the isoperimetric inequality \eqref{eq:iso} for convex bodies.

Recall that the volume of a convex body $M\subset\R^n$ is given by the formula
\begin{equation}
\vol_n(M) = \frac{1}{n}\int_{\s} \|\theta\|_M^{-n}d\theta,
\label{eq:polar_2}
\end{equation}
where the gauge, or pseudo-norm, of $M$ is $$\|\theta\|_M=\inf\{r>0:\theta\in rM\}, \quad \theta\in\s.$$ By setting $M=\pp K$ in \eqref{eq:polar_2} and applying Jensen's inequality to the (still hidden) definition of the polar projection body, one arrives at \textbf{Petty's isoperimetric inequality} \cite{petty61_1,CMP71}, which asserts that, for every convex body $K\subset\R^n,$
\begin{equation}\vol_n(\pp K)\vol_{n-1}(\partial K)^n \geq \omega_n \left(\frac{\omega_n}{\omega_{n-1}}\right)^n,
\label{eq:petty_theorem}
\end{equation}
with equality if and only if $\Pi^\circ K$ is a dilate of the Euclidean ball. We now see that combining \eqref{eq:petty_projection} with \eqref{eq:petty_theorem} yields the classical isoperimetric inequality, \eqref{eq:iso}.

To define $\pp K$, we first discuss a generalization of \eqref{eq:surface_area}. Let $L\subset \R^n$ be a compact, convex set. Then, the mixed volume of $K$ and $L$ is given by 
\begin{equation}
\label{eq:mixed_volume}
\begin{split}
    V_n(K[n-1],L):=\frac{1}{n}\lim_{\epsilon \to 0}\frac{\vol_n(K+\epsilon L)-\vol_n(K)}{\epsilon}
    =\frac{1}{n}\int_{\s}h_L(u)d\sigma_K(u).
\end{split}
\end{equation}
Here, $h_L(u)=\sup_{y\in L}\langle y,u\rangle$ is the support function of $L$ and $\sigma_K$ is the surface area measure of $K$: for $E\subset \s$ Borel,
\[
\sigma_K(E)=\mathcal{H}^{n-1}\left(\left\{y\in\partial K: K \text{ has an outer-unit normal at }y \text{ in } E\right\}\right),
\]
where $\mathcal{H}^{n-1}$ is the $(n-1)$-dimensional Hausdorff measure on $\partial K$. While on the topic, we mention that the mixed volumes satisfy \textit{Minkowski's first inequality}: for two convex bodies $K,M\subset \R^n,$
\begin{equation}V_n(K[n-1],M)^n \geq \vol_n(K)^{n-1}\vol_n(M),
\label{eq:min_first}
\end{equation}
with equality if and only if $K$ and $M$ are homothetic. Setting $M=\B$ yields \eqref{eq:iso}.

For $a,b\in\R^n,$ we denote by $[a,b]=\{(1-\lambda)a+\lambda b:\lambda\in[0,1]\}$ the line segment connecting $a$ and $b$. The polar projection body $\pp K$ of a convex body $K\subset \R^n$ is the unique origin-symmetric convex body whose gauge is given by
\[
\|\theta\|_{\pp K} = nV_n\left(K[n-1],\left[-\frac{\theta}{2},\frac{\theta}{2}\right]\right)=\frac{1}{2}\int_{\s}|\langle \theta,u \rangle|d\sigma_K(u), \quad \theta\in\s.
\]
Geometrically, $nV_n\left(K[n-1],\left[-\frac{\theta}{2},\frac{\theta}{2}\right]\right)$ equals $\vol_{n-1}(P_{\theta^\perp} K)$, where $P_{\theta^\perp} K$ is the orthogonal projection of $K$ onto $\theta^\perp$, the hyperplane through the origin orthogonal to $\theta$. Using the definition of $\pp K$, one can verify that $K\mapsto \vol_n(K)^{n-1}\vol_n(\pp K)$ is an affine invariant functional on the set of convex bodies.

We next discuss stability of Petty's projection inequality: if $K$ is ``close", in some quantitative sense, to an ellipsoid, how far is one from equality in \eqref{eq:petty_projection}? The volume ratio of a convex body $K$ is given by
\begin{equation}
    \mathrm{v.r}(K) = \left(\frac{\vol_n(K)}{\vol_n(\mathcal{E})}\right)^\frac{1}{n},
\end{equation}
where $\mathcal{E}$ is the maximal volume ellipsoid contained in $K$. The ellipsoid $\mathcal{E}$ exists, and is unique, as shown in a classical theorem by Fritz John. Using $K\supseteq \mathcal{E}$, we have $\mathrm{v.r}(K)\geq 1$. Supposing that $K$ is translated so it has center of mass at the origin, one has $K\subset n \mathcal{E}$, which yields, by the translation invariance of the Lebesgue measure, that $\mathrm{v.r}(K) \leq n$ for all convex bodies $K$.

We say $K$ is in John position if $\mathcal{E}$ is $\B$. For an arbitrary convex body $K\subset \R^n$, there exists, by John's theorem, a unique, up to composition with an orthogonal transformation, affine transformation $A$ such that $AK$ in John position. Therefore, 
\begin{equation}
\label{eq:john_volume}
    \mathrm{v.r}(K) = \left(\frac{\vol_n(A K)}{\vol_n(\B)}\right)^\frac{1}{n}.
\end{equation}
Using $\mathrm{v.r}(K)$, it was shown by Alonso-Guti\'{e}rrez \cite{AG14} that
\begin{equation}
\label{eq:reverse_petty}
\mathrm{v.r}(K)^{-n}\left(\frac{\omega_n}{\omega_{n-1}}\right)^n\leq \vol_n(K)^{n-1}\vol_{n}(\pp K).\end{equation}
This should be compared to Zhang's projection inequality, proven by Zhang in 1991 \cite{Zhang91}:
\begin{equation}\label{e:Zhang_ineq}
\frac{1}{n^n} \binom{2n}{n}\leq\vol_n(K)^{n-1}\vol_n(\Pi^\circ K).
\end{equation}
There is equality if and only if $K$ is an $n$-dimensional simplex. B\"or\"oczky also considered stability for Petty's projection inequality using the Banach-Mazur distance instead of the volume ratio \cite{BK13}.
\section{Higher-Order Projection Bodies}

The purpose of this note is to establish a version of \eqref{eq:reverse_petty} for a recent generalization of polar projection bodies. To motivate this generalization, we first recall the \textit{covariogram function} of a convex body $K$ in $\R^n$:
	\begin{equation}\label{e:covario}
	    \cov(x)=\vol_n\left(K\cap(K+x)\right), \quad x\in\R^n.
	\end{equation}
 The covariogram function is a vital tool in geometric tomography. For example, the fact that it is $(1/n)$-concave on its support, which is the difference body of $K$, is the key in Chakerian's \cite{Cha67} proof of the Rogers-Shephard inequality \cite{RS57}. The relevance of the covariogram function to our current considerations is the following result by Matheron \cite{MA}: for $\theta\in\s$, one has
 $$\frac{\mathrm{d}\cov(r\theta)}{\mathrm{d}r}\bigg|_{r=0^+}=-\|\theta\|_{\pp K}.$$
In fact, the above variational formula coupled again with the $(1/n)$-concavity of the covariogram function can be used to prove \eqref{e:Zhang_ineq}, see \cite{GZ98}. The reader is recommended to see the excellent survey by Bianchi \cite{GB23} for more on the intricacies of the covariogram function.

In 1970, Schneider \cite{Sch70} introduced for $m\in\N$ the $m$th-order difference body and established the $m$th-order Rogers-Shephard inequality. Along the way, he defined the $m$th-order covariogram function: using the notation $\bar x=(x_1,\dots,x_\hid)$, for $x_i\in\R^n$,
\begin{equation}
\Cov(\overline{x})=\vol_n\left(K\cap\bigcap_{i=1}^\hid (x_i+K)\right), \quad \bar x \in (\R^n)^m.
\label{eq_vol_ell_co}
\end{equation}
Henceforth, we identify $(\R^n)^m$ with $\Rnhi$. It is natural to determine the variation of $\Cov$. To this end, we define a polytope in $\R^n$ from a unit vector in $\S$ in the following way: for $\bar\theta=(\theta_1,\dots,\theta_\hid)\in\S,$ set $$C_{\bar\theta}=\conv_{1\leq i \leq \hid}[o,\theta_i],$$ where $o\in\R^n$ is the origin and $$\conv_{1\leq i \leq \hid}(A_i)=\left\{\sum_{i=1}^m\lambda_i x_i: x_i\in A_i, \lambda_i\in [0,1], \sum_{i=1}^m\lambda_i=1\right\}$$ denotes the convex hull of convex sets $A_1,\dots,A_\hid.$ We will need later that  $$h_{C_{-\bar\theta}}(v)=\max_{1 \leq i \leq \hid} \langle \theta_i, v \rangle_{-}, \qquad v\in \R^n, \quad\bar\theta=(\theta_1,\dots,\theta_m)\in\S,$$ where $a_-=\max\{0,-a\}, a \in \R.$

We, working with Haddad, Putterman, Roysdon, and Ye \cite{HLPRY25}, showed the following. 
\begin{theorem} \label{t:variationalformula}
Fix $n,m\in\N$. Let $K\subset\R^n$ be a convex body. Then, for every direction $\bar{\theta} = (\theta_1,\dots,\theta_{\hid})$ $\in \S$:
\[
\frac{\mathrm{d}}{\mathrm{d}r} \Cov (r\bar{\theta}) \bigg|_{r=0^+}
= -nV_n(K[n-1],C_{-\bar\theta})=-\int_{\s} \max_{1 \leq i \leq \hid} \langle \theta_i, u \rangle_{-}d\sigma_K(u).\]
\end{theorem}
This theorem then motivated the following generalization of the polar projection body. 
\begin{definition}
\label{def:higher_proj}
    Fix $n,\hid\in\N$. Let $K\subset\R^n$ be a convex body. Then, its \textit{$\hid$th-order polar projection body} $\PP K$ is the $n\hid$-dimensional convex body containing the origin in its interior whose gauge function is defined, for $\bar{\theta}=(\theta_1,\dots,\theta_{\hid})\in\S$, as
\begin{align*}\|\bar{\theta}\|_{\PP K}= nV_n(K[n-1],C_{-\bar\theta})
=\int_{\s} \max_{1 \leq i \leq \hid} \langle \theta_i, u \rangle_{-}d\sigma_K(u).
\end{align*}
\end{definition}
With this definition in hand, we can rephrase Theorem~\ref{t:variationalformula} as, for all $\bar\theta\in\S,$ $$\frac{\mathrm{d}}{\mathrm{d}r} \Cov (r\bar{\theta}) \bigg|_{r=0^+} = -\|\bar\theta\|_{\PP K}.$$
Before moving on, let us mention some properties of $\PP K$, to better familiarize ourselves. First, note that $\Pi^{\circ,1} K=\pp K$. The translation invariance of mixed volumes shows that $\PP (K+x) = \PP K$ for every $x\in\Rn.$ For $u\in\s,$ let $u_j=(o,\dots,o,u,o,\dots,o)\in\S$, where $u$ is in the $j$th coordinate. We then see that
$$\|u_j\|_{\PP K}=nV_n(K[n-1],[o,-u])=\|u\|_{\pp K}.$$
Since this is independent of $j$, this shows that the intersection of $\PP K$ with any of the $\hid$ copies of $\R^n$ is $\pp K$, but, by taking $u^m=\frac{1}{\sqrt{m}}(u,\dots,u)$, the fact that $$\|u^m\|_{\PP K}=nV_n\left(K[n-1],\left[o,-\frac{u}{\sqrt{m}}\right]\right)=\frac{1}{\sqrt{m}}\|u\|_{\pp K},$$ yields $\PP K$ is not merely the convex hull of said sections. 

The case of $K=\B$ also deserves special attention. We recall that the mean width of a compact, convex set $L\subset\R^n$ is given by
\begin{align}
\label{eq:mean_wdith}
W_n(L)=\frac{1}{n\omega_n}\int_{\s}h_L(\theta)d\theta=\frac{1}{\omega_n}V_n(\B[n-1],L),
\end{align}
which is precisely the average value of $h_L$ with respect to the Haar measure on $\s$. The gauge of $\PP\B$ is then
\[\|\bar x\|_{\PP\B} = n\omega_n W_n(C_{\bar x}).\]
Therefore, $\PP \B$ is a probabilistic object whose precise shape, even its volume, is difficult to determine. In fact, using \eqref{eq:polar_2}, we see that $$\vol_{nm}(\PP \B) = \frac{\omega_{nm}}{(n\omega_n)^{nm}}\E[W_n(C_{\bar\Theta})^{-nm}],$$
where the expectation is taken with respect to the Haar measure on $\S$.

Motivated by Theorem~\ref{t:variationalformula}, we, working with Haddad, Putterman, Roysdon, and Ye, showed the following extension of Petty's projection and Zhang's projection inequalities \cite{HLPRY25}.
\begin{theorem}[Zhang's projection and Petty's projection inequalities for $m$th-order projection bodies]
\label{t:pettyprojectioninequality}
    Fix $m,n\in\N$ and let $K\subset\R^n$ be a convex body. Then,
    \[\frac{1}{n^{n\hid}}\binom{n\hid+n}{n}\leq \vol_n(K)^{n\hid-\hid}\vol_{n\hid}\left( \PP K \right) \leq \frac{\omega_{nm}}{n^{nm}\omega_n^{\hid}}\E[W_n(C_{\bar\Theta})^{-nm}].\]
    There is equality in the first inequality if and only if $K$ is an $n$-dimensional simplex, and there is equality in the second inequality if and only if $K$ is an ellipsoid.
\end{theorem}
The first step in proving Theorem~\ref{t:pettyprojectioninequality} was to verify $K\mapsto \vol_n(K)^{n\hid-\hid}\vol_{n\hid}(\PP K)$ is an affine invariant functional.

\begin{proposition}[Petty Product for $m$th-Order Projection Bodies]
\label{p:affine_invar_Zhang}
    Fix $n,\hid\in\mathbb{N}$, Then, the functional $$K\mapsto \vol_n(K)^{n\hid-\hid}\vol_{n\hid}(\PP K)$$ on the set of convex bodies in $\R^n$ is invariant under affine transformations.    
\end{proposition}

\noindent Along the way, the following generalization of Petty's isoperimetric inequality, \eqref{eq:petty_theorem}, was established.
\begin{theorem}
\label{t:petty_classic}
    Fix $n,m\in\N$. Let $K\subset\R^n$ be a convex body. Then, it holds
    \begin{equation}
    \label{eq:petty_classic}
    \vol_{n\hid}(\PP K)\vol_{n-1}(\partial K)^{n\hid} \geq  \omega_{nm}\E[W_n(C_{\bar\Theta})^{-nm}],
    \end{equation}
There is equality if and only if $\pp K$ $(m=1)$ or $K$ ($m\geq 2$) is a Euclidean ball.
\end{theorem}

The equality conditions to Theorem~\ref{t:petty_classic} in the original work \cite{HLPRY25} were left implicit by the proof; we now take a moment to elaborate on them, as they are crucial to our considerations herein. We denote by $\mathcal{O}(n)$ the orthogonal group on $\R^n$.
\begin{proof}[Proof of the equality characterization of Theorem~\ref{t:petty_classic}]
It was shown in \cite[Proof of Theorem 1.5, pg. 36]{HLPRY25} that there is equality in \eqref{eq:petty_classic} if and only if $K$ is a convex body such that, for every $O\in \mathcal{O}(n)$ on $\s$,
\begin{equation}
\PP (O\cdot K) = \PP K.
\label{eq:equality_case}
\end{equation}
First suppose that $\hid =1$. Then, since $\pp (O\cdot K) = O \pp K$ \cite[Theorem 4.1.5]{gardner_book}, we deduce that $\pp K$ is rotationally invariant, and, therefore, a centered Euclidean ball. Now, suppose $\hid \geq 2$. It was shown in \cite[Proposition 3.5]{HLPRY25} that, for such $m$, if $K,M$ are convex bodies such that $\PP M = \PP K$, then $M=K+b$ for some $b\in \R^n$. Applying this to \eqref{eq:equality_case}, we deduce for every $O\in \mathcal{O}(n)$ that there exists $b_O\in \R^n$ such that 
\begin{equation}O\cdot K =K+b_O.
\label{eq:equality_sets}
\end{equation} 
By picking $O=-\operatorname{Id}_n$, where $\operatorname{Id}_n$ is the $n\times n$ identity matrix, we deduce $- K =K+b_{-\operatorname{Id}_n}$. Therefore, $K$ is symmetric about a point. Let us call this point of symmetry $c$, and define $K_0=K-c$. Then, $K_0$ has center of mass at the origin. Inserting $K_0$ into \eqref{eq:equality_sets}, we obtain
\begin{equation}
O\cdot K_0 =K_0+(b_O+c-Oc).
\label{eq:equality_sets_2}
\end{equation} 
Since the left-hand side of \eqref{eq:equality_sets_2} has center of mass at the origin, so does the right-hand side. We deduce that $b_O=Oc-c$, i.e. $O\cdot K_0=K_0$ for every $O\in \mathcal{O}(n)$. Thus, $K_0$ is a dilate of $\B$. We conclude from the fact that $K=K_0 +c$ that $K$ is a Euclidean ball. 
\end{proof}

The first purpose of this note is to ``complete" the above story by showing stability in the $m$th-order setting. Our approach is based on \cite{GP99,Ball91,AG14}. 
\begin{theorem}
\label{t:stab}
    Fix $m,n\in\N$. Let $K\subset\R^n$ be a convex body. Then,
    $$\vol_n(K)^{n\hid-\hid}\vol_{n\hid}(\PP K) \geq \mathrm{v.r}(K)^{-nm} \frac{\omega_{nm}}{n^{nm}\omega_n^{\hid}}\E[W_n(C_{\bar\Theta})^{-nm}].$$
\end{theorem}
Theorem~\ref{t:stab} reduces to \eqref{eq:reverse_petty} when $\hid=1$. To prove Theorem~\ref{t:stab}, we need another definition. For a convex body $K\subset\R^n$, its minimal isoperimetric ratio is given by
\begin{equation}
\label{eq:min_iso}
    \partial_K:=\min_{T\in GL_n(\R)}\frac{\vol_{n-1}(\partial (TK))}{\vol_n(TK)^\frac{n-1}{n}}.
\end{equation}
Here, $GL_n(\R)$ is the set of $n\times n$, non-singular matrices with real entries. By the isoperimetric inequality \eqref{eq:iso}, $\partial_K>0$; Petty showed that the minimum in \eqref{eq:min_iso} is obtained and that the matrix that obtains said minimum is unique up to composition with orthogonal transformations (see, e.g. \cite[Theorem 2.3.1]{AGA}). We say $K$ is in minimal surface area position if the minimum is obtained at the identity. Theorem~\ref{t:stab} is now an immediate consequence of the following two facts.
The first is a corollary of Theorem~\ref{t:petty_classic} and extends on the $m=1$ case from \cite{GP99}.
\begin{lemma}
    Fix $m,n\in\N$. Let $K\subset\R^n$ be a convex body. Then,
    $$\frac{\vol_n(K)^{n\hid-\hid}\vol_{n\hid}(\PP K)}{ \omega_{nm}\E[W_n(C_{\bar\Theta})^{-nm}]} \geq \partial_K^{-nm}.$$
    There is equality if and only if $K$ is in minimal surface area position and $\pp K$ ($m=1$) or $K$ ($m\geq 2$) is a Euclidean ball.
\end{lemma}
\begin{proof}
    Begin by rewriting the inequality in Theorem~\ref{t:petty_classic} as
    $$\frac{\vol_n(K)^{n\hid-\hid}\vol_{n\hid}(\PP K)}{ \omega_{nm}\E[W_n(C_{\bar\Theta})^{-nm}]} \geq \left(\frac{\vol_n(K)^\frac{n-1}{n}}{\vol_{n-1}(\partial K)}\right)^{nm}.$$
    In the above inequality, replace $K$ with its affine image in minimal surface area position and conclude with Proposition~\ref{p:affine_invar_Zhang}. The equality conditions are inherited from Theorem~\ref{t:petty_classic}.
\end{proof}
The last fact is the following lemma, essentially proven in \cite{AG14}.
\begin{lemma}
    Let $K\subset\R^n$ be a convex body. Then, 
    $\partial_K \leq n\omega_n^\frac{1}{n}\mathrm{v.r}(K).$
\end{lemma}
\begin{proof}
Let $A$ be any of the affine transformations that place $K$ in John position. It was shown by K. Ball \cite{Ball91} that $\vol_{n-1}(\partial (A K))\leq n \vol_n(A K)$.
    Therefore, from the fact that $\partial_K$ defined in \eqref{eq:min_iso} is a minimum, we have from the translation invariance of the Lebesgue measure
    $$\partial_K \leq \frac{\vol_{n-1}(\partial (A K))}{\vol_n(A K)^\frac{n-1}{n}} \leq n\vol_n(A K)^\frac{1}{n}.$$
    We conclude using \eqref{eq:john_volume}.
\end{proof}

Before concluding our discussion of the $m$th-order polar projection bodies, we mention that there are other generalizations of the polar projection body and Petty's projection inequality. It would take too much space to mention them all in detail, so we simply mention the $L^p$ \cite{LYZ00} and Orlicz \cite{LYZ10} cases by Lutwak, Yang, and Zhang and the asymmetric $L^p$ case by Haberl and Schuster \cite{HS09}. We, working again with Haddad, Putterman, Roysdon, and Ye, later considered an $L^p$ version of the $m$th-order polar projection bodies in \cite{HLPRY25_2}. Very recently, Ye, Zhou, and Zhang \cite{YZZ25} considered Orlicz versions. A main motivation for establishing generalizations of the Petty projection inequality is that it implies sharp affine Sobolev-type inequalities \cite{GZ99,LYZ02,HS2009,TW12}, {P}\'olya-{S}zeg\"{o} principles \cite{CLYZ09,HSX12,TW13,LRZ24}, and moment-entropy inequalities \cite{LLYZ13,NVH19,LYZ07,LYZ04_3,LYZ05_2,DL25}. 

\section{Mixed Projection Bodies}
The mixed volumes can be generalized further. To discuss this generalization, we must describe the mixed volumes from a different perspective. Let $K_1, K_2, \dots, K_r \subset \R^n$ be compact, convex sets in $\R^n$ and $\lambda_1, \dots, \lambda_r \ge 0$. Then, it turns out that the volume of the Minkowski summation
$\lambda_1 K_1 +\lambda_2 K_2+\dots + \lambda_r K_r$ is a homogeneous polynomial in the variables $\lambda_1,\dots, \lambda_r$ of degree $n$:
$$
\vol_n(\lambda_1 K_1 +\lambda_2 K_2+\dots + \lambda_r K_r)=\sum\limits^r_{i_1, i_2, \dots, i_n =1} V_n(K_{i_1}, \dots, K_{i_n}) \lambda_{i_1}\lambda_{i_2}\dots \lambda_{i_n}.
$$
The coefficients $V_n(K_{i_1}, \dots, K_{i_n})$ are precisely the mixed volume of $K_{i_1}, \dots, K_{i_n}.$ The mixed volumes are invariant under translation and permutation of their entries. We use the notation $V_n(K_1, \dots, K_j, K\dots, K)=V_n(K_1, \dots, K_j, K[n-j]).$
To determine a formula for the mixed volumes, we note that the surface area measure of a Minkowski sum can also be expanded as a polynomial (in the same way). In general, the mixed volume of a collection $K_1,\dots,K_{n-1},L$ of compact, convex sets, all in $\R^n$, has the integral formula
\[
V_n(K_1,\dots,K_{n-1},L)=\frac{1}{n}\int_{\s}h_L(u)d\sigma_{K_1,\dots,K_{n-1}}(u),
\]
with $\sigma_{K_1,\dots,K_{n-1}}$ being the mixed area measure of $K_1,\dots,K_{n-1}$ \cite[Theorem 5.1.7, eq. 5.18, pg. 280]{Sh1}. 
 With this definition, Lutwak \cite{LE85} introduced \textit{mixed polar projection bodies}: letting $K_1,\dots,K_{n-1}\subset \R^n$ be convex bodies, their mixed polar projection body $\pp (K_1,\dots,K_{n-1})$ is given by its gauge, for $\theta\in \s,$
\begin{align*}
\|\theta\|_{\pp (K_1,\dots,K_{n-1})}=nV_n\left(K_1,\dots,K_{n-1},\left[-\frac{\theta}{2},\frac{\theta}{2}\right]\right)
=\frac{1}{2}\int_{\s}|\langle \theta,u \rangle|d\sigma_{K_1,\dots,K_{n-1}}(u).
\end{align*}
These were extended to the $L^p$ setting by Wang and Leng \cite{LW07}. In this section, we extend the definition of $\pp (K_1,\dots,K_{n-1})$ and Lutwak's results for them to the $m$th-order setting.
\begin{definition}
\label{def:mixed}
    Fix $n\geq 2, m\geq 1$. Let $K_1,\dots,K_{n-1}\subset\R^n$ be convex bodies. Then, we define their $m$th-order mixed polar projection body as the body $\PP(K_1,\dots,K_{n-1})\subset \R^{nm}$ given by the gauge, for $\bar\theta=(\theta_1,\dots,\theta_m)\in \S$,
    \[
    \|\bar\theta\|_{\PP(K_1,\dots,K_{n-1})}=nV_n(K_1,\dots,K_{n-1},C_{-\bar\theta}).
    \]
\end{definition}
Our result for this section is the following theorem, extending \cite[Theorem 3.8]{LE85}.
\begin{theorem}
\label{t:mixed}
    Fix $n\geq 2,m\geq 1$. Let $K_1,\dots,K_{n-1}\subset \R^{n}$ be convex bodies. Then,
    \begin{align*}
    \vol_{nm}(\PP(K_1,\dots,K_{n-1}))^{n-1}\leq \prod_{i=1}^{n-1}\vol_{nm}(\PP K_i).
    \end{align*}
    There is equality if and only if the $K_i$ are homothetic to each other.
\end{theorem}
Combining Theorem~\ref{t:mixed} with Theorem~\ref{t:pettyprojectioninequality}, we obtain the following corollary.
\begin{corollary}
    Fix $n\geq 2, m\geq 1$. Let $K_1,\dots,K_{n-1}\subset\R^n$ be convex bodies. Then,
    \begin{align*}
    \left(\prod_{i=1}^{n-1}\vol_n(K_i)\right)^m\vol_{nm}(\PP(K_1,\dots,K_{n-1}))
    \leq \frac{\omega_{nm}}{n^{nm}\omega_n^{\hid}}\E[W_n(C_{\bar\Theta})^{-nm}],
    \end{align*}
    with equality if and only if the $K_i$ are homothetic ellipsoids.
\end{corollary}
To prove Theorem~\ref{t:mixed}, we need two inequalities. The first is the Aleksandrov-Fenchel inequality, which asserts that, if $K_1,K_2\dots,K_n$ are compact, convex sets, and we set $\mathfrak{K}=(K_3,\dots,K_n)$, then
\begin{equation}
    V_n(K_1,K_2,\mathfrak{K})^2\geq V_n(K_1,K_1,\mathfrak{K})V_n(K_2,K_2,\mathfrak{K}).
    \label{eq:AF}
\end{equation}
We also need the following consequence of Aleksandrov-Fenchel (see \cite[Section 7.4, eq. 7.65, pg. 401]{Sh1}): let $K_1$ and $K_2$ be two compact, convex sets and let $\mathfrak{K}=(K_{j+2},\dots,K_n)$ be a collection of $(n-1-j)$ compact, convex sets. Then, for $j=1,\dots,n-2,$
\begin{equation}
    V_n(K_1[j],K_2,\mathfrak{K})^{j+1}\geq V_n(K_1[j+1],\mathfrak{K})^{j}V_n(K_2[j+1],\mathfrak{K}).
    \label{eq:AF_2}
\end{equation}
With these inequalities in mind, Theorem~\ref{t:mixed} follows as an immediate application of the following two lemmas. We introduce the notation
\begin{equation}
\label{eq:concise_notation}
\PP(K_1,K_2)=\PP(K_1,K_1\dots,K_1,K_{2}).\end{equation}
\begin{lemma}
\label{l:mixed_1}
    Fix $n\geq 2, m\geq 1$. Let $K_1,K_2\subset\R^n$ be convex bodies. Then, 
    \[
    \vol_{nm}(\PP(K_1,K_2))^{n-1} \leq \vol_{nm}(\PP K_1)^{n-2} \vol_{nm}(\PP K_2) .
    \]
    There is equality if and only if $K_1$ and $K_2$ are homothetic.
\end{lemma}
\begin{proof}
From Definition~\ref{def:mixed} and \eqref{eq:AF_2}, applied to the case when $j=n-2$ and $\mathfrak{K}=\{C_{-\bar\theta}\},$ we have, for every $\bar\theta\in\S$,
\begin{equation}
\label{eq:AF_2.5}
    \|\bar\theta\|^{n-1}_{\PP(K_1,K_2)}\geq \|\bar\theta\|^{n-2}_{\PP K_1}\|\bar\theta\|_{\PP K_2}.
\end{equation}
    Next, applying \eqref{eq:polar_2}, H\"older's inequality and \eqref{eq:AF_2.5}, we have
    \begin{equation}
    \begin{split}
        &(nm\vol_{nm}(\PP(K_1,K_2)))^{n-1} = \left(\int_{\S}\|\bar\theta\|_{\PP(K_1,K_2)}^{-nm}d\bar\theta\right)^{n-1}
        \\
        &\leq \left(\int_{\S}\left(\|\bar\theta\|_{\PP K_1}^{-nm} \right)^\frac{n-2}{n-1}\left(\|\bar\theta\|_{\PP K_2}^{-nm}\right)^\frac{1}{n-1}d\bar\theta\right)^{n-1}
        \\
        &\leq \left(\int_{\S} \|\bar\theta\|_{\PP K_1}^{-nm} d\bar\theta\right)^{n-2}\left(\int_{\S} \|\bar\theta\|_{\PP K_2}^{-nm} d\bar\theta\right)
        \\
        &=(nm)^{n-1}\vol_{nm}(\PP K_1)^{n-2}\vol_{nm}(\PP K_2).
    \end{split}
    \end{equation}
    Suppose $K_1$ and $K_2$ are such that equality holds. Then, we must have equality in the use of \eqref{eq:AF_2.5} for almost all $\bar\theta\in\S$. From the equality condition in H\"older's inequality, and the fact these functions are continuous, we have equality in \eqref{eq:AF_2.5} for all $\bar\theta\in\S$. By considering only vectors of the form $\bar\theta=(\theta,o,\dots,o)$, we deduce that, for all $\theta\in\s$,
    \begin{align*}
    V_{n-1}(P_{\theta^\perp }K_1[n-2],P_{\theta^\perp }K_2)^{n-1} &=V_n(K_1[n-2],K_1,[o,\theta])^{n-1}
    \\
    &=V_n(K_1[n-1],[o,\theta])^{n-2}V_n(K_2[n-1],[o,\theta])
    \\
    &=\vol_{n-1}(P_{\theta^\perp }K_1)^{n-2}\vol_{n-1}(P_{\theta^\perp }K_2).
    \end{align*} 
    The first and third equalities follow from \cite[Theorem 5.3.1]{Sh1}.
    This is equality in Minkowski's first inequality, \eqref{eq:min_first}, in dimension $(n-1)$. Thus, $P_{\theta^\perp }K_1$ is homothetic to $P_{\theta^\perp }K_2$ for all $\theta\in\s$. It is well-known that this means $K_1$ is homothetic to $K_2$ (see \cite[pg. 101 and the notes on pgs. 126-127]{gardner_book}).
\end{proof}

\begin{lemma}
\label{l:mixed_2}
    Fix $n\geq 2, m\geq 1$. Let, for $i=1,\dots,n-1$, $K_i\subset\R^n$ be convex bodies in $\R^n$ and set $\mathfrak{K}=(K_3,\dots,K_{n-1})$. Then, 
    \[
    \vol_{nm}(\PP(K_1,K_2,\mathfrak{K}))^2 \leq \vol_{nm}(\PP (K_1,K_1,\mathfrak{K})) \vol_{nm}(\PP (K_2,K_2,\mathfrak{K}) ).
    \]
    There is equality when the $K_i$ are homothetic to each other.
\end{lemma}
\begin{proof}
From Definition~\ref{def:mixed} and \eqref{eq:AF}, we have
\begin{equation}
\label{eq:AF_3}
    \|\bar\theta\|^{nm}_{\PP(K_1,K_2,\mathfrak{K})}\geq \|\bar\theta\|^{\frac{nm}{2}}_{\PP (K_1,K_1,\mathfrak{K})}\|\bar\theta\|_{\PP (K_2,K_2,\mathfrak{K})}^\frac{nm}{2}.
\end{equation}
    Next, applying \eqref{eq:polar_2}, H\"older's inequality, and \eqref{eq:AF_3}, we have
    \begin{equation}
    \begin{split}
        &(nm)\vol_{nm}(\PP(K_1,K_2,\mathfrak{K})) = \int_{\S}\|\bar\theta\|_{\PP(K_1,K_2,\mathfrak{K})}^{-nm}d\bar\theta
        \\
        &\leq \int_{\S}\left(\|\bar\theta\|_{\PP (K_1,K_1,\mathfrak{K})}^{-nm} \right)^\frac{1}{2}\left(\|\bar\theta\|_{\PP (K_2,K_2,\mathfrak{K})}^{-nm}\right)^\frac{1}{2}d\bar\theta
        \\
        &\leq \left(\int_{\S} \|\bar\theta\|_{\PP (K_1,K_1,\mathfrak{K})}^{-nm} d\bar\theta\right)^\frac{1}{2}\left(\int_{\S} \|\bar\theta\|_{\PP (K_2,K_2,\mathfrak{K})}^{-nm} d\bar\theta\right)^\frac{1}{2}
        \\
        &=(nm)\vol_{nm}(\PP (K_1,K_1,\mathfrak{K}))^{\frac{1}{2}}\vol_{nm}(\PP (K_2,K_2,\mathfrak{K}))^\frac{1}{2}.
    \end{split}
    \end{equation}
    Arguing like in the proof of Lemma~\ref{l:mixed_1}, equality occurs when, for all $\theta\in\s$,
    \[
    V_{n-1}(P_{\theta^\perp}K_1,P_{\theta^\perp}K_2,\mathfrak{K}_\theta)^2 = V_{n-1}(P_{\theta^\perp}K_1,P_{\theta^\perp}K_1,\mathfrak{K}_\theta)V_{n-1}(P_{\theta^\perp}K_2,P_{\theta^\perp}K_2,\mathfrak{K}_\theta),
    \]
    where $\mathfrak{K}_\theta=(P_{\theta^\perp}K_3,\dots,P_{\theta^\perp}K_{n-1})$. 
    
    This occurs, for example, when each $P_{\theta^\perp}K_i$ is homothetic. If this special case is true for all $\theta\in\s$, then we must have that each $K_i$ is homothetic.
\end{proof}
We are unable to obtain the fully equality characterization in Lemma~\ref{l:mixed_2} due to our use of the Aleksandrov-Fenchel inequality \eqref{eq:AF}, whose equality conditions are still open. Although the inequality was established over a century ago, progress towards equality characterization was made only recently; see \cite{SH23} and the references therein.

The proof of Theorem~\ref{t:mixed} is the same in spirit as the passage from \eqref{eq:AF} to \eqref{eq:AF_2}, via repeated applications of Lemma~\ref{l:mixed_2} with final applications of Lemma~\ref{l:mixed_1}. We include the details of this argument for completeness. 

To this end, we isolate the applications of Lemma~\ref{l:mixed_2} in the following two lemmas. To state them concisely, we introduce the notation, for $j=1,\dots,n-1$ and $K_1,K_2,\dots,K_j\subset \R^n$ convex bodies, $$\PP(K_{1}[i_1],\dots,K_j[i_j]),\quad \text{ where }\sum_{k=1}^ji_k=n-1,$$ to represent that, for $k=1,\dots,j$, the body $K_k$ is repeated $i_k$ times among the convex bodies appearing in Definition~\ref{def:mixed}. Furthermore, we say a sequence of positive numbers $(a_0,a_1,\dots,a_j)$ is \textit{log-concave} if $a_i^2\geq a_{i-1}a_{i+1}$ for $i=1,\dots,j-1$. It follows by iteration (see \cite[Section 7.4, eq. 7.61, pg. 400]{Sh1}) that 
\begin{equation}
\label{eq:log_concave_it}
a_\alpha^{k-i}\geq a_i^{k-\alpha}a_k^{\alpha-i} \qquad \text{for} \quad 0\leq i<\alpha<k\leq j.\end{equation} 
For fixed $i<k\leq j$, there is equality in \eqref{eq:log_concave_it} if and only if $a_\alpha^2=a_{\alpha-1}a_{\alpha+1}$ for $\alpha=i+1,\dots,k-1$; we say the sequence $(a_0,\dots,a_j)$ is \textit{log-affine} between $i$ and $k$.
Similarly, we say that a sequence of positive numbers $(c_0,c_1,\dots,c_j)$ is \textit{log-convex} if the sequence $a_i=c_i^{-1}$ is log-concave.

\begin{lemma}
\label{l:induction_1}
Let $n\geq 2,m\geq 1$. Then, for all convex bodies $K_1,\dots,K_{n-1}\subset \R^n$, and $j=1,\dots,n-1,$ one has, with $\mathfrak{K}=(K_{j+1},\dots,K_{n-1})$,
\begin{equation}
    \vol_{nm}\left(\PP\!(K_1[j-1],K_2,\mathfrak{K})\right)^{j}\!\!\leq\!\! \vol_{nm}\left(\PP\!(K_1[j],\mathfrak{K})\right)^{j-1}\!\vol_{nm}\left(\PP\!(K_2[j],\mathfrak{K})\right).
    \label{eq:RAF_2}
\end{equation}
There is equality when the $K_i$ are homothetic to each other.
\end{lemma}
\begin{proof}
The proof of \eqref{eq:RAF_2} is the exact same as the proof of deriving \eqref{eq:AF_2} from \eqref{eq:AF}, but with all inequalities reversed. For completeness, we reproduce this argument from Schneider's textbook \cite[Section 7.4, eq. 7.65, pg. 401]{Sh1}. Fix $j\in \{1,2\dots,n-1\}$. For $i=0,\dots,j$ define the constants
\[
c_i = \vol_{nm}\left(\PP(K_1[j-i],K_2[i],\mathfrak{K})\right).
\]
By Lemma~\ref{l:mixed_2}, the sequence $(c_0,\dots,c_j)$ is log-convex. Hence, by applying the iterated formula \eqref{eq:log_concave_it}, we have $$c_\alpha^{k-i}\leq c_i^{k-\alpha}c_k^{\alpha-i}, \quad \text{for} \quad 0\leq i<\alpha<k\leq j.$$
We conclude the proof of the inequality by picking $i=0$, $\alpha=1$ and $k=j$. For the final claim, if the $K_i$ are homothetic, then the sequence $(c_0,\dots,c_j)$ is log-affine. We conclude.
\end{proof}
Notice that Lemma~\ref{l:induction_1} with $j=n-1$ is precisely Lemma~\ref{l:mixed_1}. That is, the reverse Aleksandrov-Fenchel-type inequality in Lemma~\ref{l:mixed_2} \textit{implies} by iteration the reverse Minkowski-type inequality in Lemma~\ref{l:mixed_1}. This is not surprising, and, indeed, is expected. However, what Lemma~\ref{l:mixed_1} provides us, that Lemma~\ref{l:induction_1} does not, is the sharp equality characterization.

\begin{lemma} 
\label{l:induction_2}
Fix $n\geq 2, m\geq 1$. For convex bodies $K_1,\dots,K_{n-1}\subset \R^n$ and any $j \in \{2, \dots, n-1\}$, we have, by setting $\mathfrak{K}=(K_{j+1},\dots,K_{n-1})$,
\begin{equation}
\label{eq:induction}
\vol_{nm}\left(\PP(K_1, \dots, K_j, \mathfrak{K})\right)^j \leq \prod_{i=1}^j \vol_{nm}\left(\PP(K_i[j], \mathfrak{K})\right).
\end{equation}
There is equality when the $K_i$ are homothetic to each other.
\end{lemma}
\begin{proof}
    The proof of \eqref{eq:induction} is the exact same as the proof of \cite[Section 7.4, eq. 7.64, pg. 401]{Sh1}, but with all inequalities reversed. For completeness, we reproduce this argument. 

We begin by noticing that the case $j=2$ is Lemma~\ref{l:mixed_2}. We proceed by induction. Suppose that we have \eqref{eq:induction} for every $k=2,\dots,j-1$: for every collection of convex bodies $L_1,\dots,L_{n-1}\subset\R^n$, we have
\begin{equation}
\label{eq:induction_2}
\vol_{nm}\left(\PP(L_1, \dots, L_{k}, \mathfrak{L})\right)^k \leq \prod_{i=1}^k \vol_{nm}\left(\PP(L_i[k], \mathfrak{L})\right),
\end{equation}
where $\mathfrak{L}=(L_{k+1},\dots,L_{n-1})$, and, furthermore, the $L_i$ being homothetic to each other implies equality in \eqref{eq:induction_2}. Then, for every collection of convex bodies $K_1,\dots,K_{n-1}\subset\R^n$
\begin{equation}
\label{eq:induction_3}
\vol_{nm}\left(\PP(K_1, \dots, K_{j-1},K_j, \mathfrak{K})\right)^{j-1} \leq \prod_{i=1}^{j-1} \vol_{nm}\left(\PP(K_i[j-1],K_j, \mathfrak{K})\right),
\end{equation}
where $\mathfrak{K}=(K_{j+1},\dots,K_{n-1})$ and we applied \eqref{eq:induction_2} with $k=j-1$, $L_i=K_i$ for $i=1,\dots,j-1$ and $\mathfrak{L}=(K_j,\mathfrak{K})$. 
Next, we apply Lemma~\ref{l:induction_1} and obtain, for $i=1,\dots,j-1,$
\begin{equation}
\label{eq:inner_step}
\vol_{nm}\left(\PP\!(K_i[j-1], K_j, \mathfrak{K})\right) \!\leq\!\! \vol_{nm}\left(\PP\!(K_i[j], \mathfrak{K})\right)^\frac{j-1}{j} \!\!\vol_{nm}\left(\PP\!(K_j[j], \mathfrak{K})\right)^\frac{1}{j}.
\end{equation}
Substituting \eqref{eq:inner_step} back into \eqref{eq:induction_3}, the inequality becomes:
\begin{equation}
\vol_{nm}\left(\PP(K_1, \dots, K_j, \mathfrak{K})\right)^{j-1} \leq \prod_{i=1}^{j} \vol_{nm}\left(\PP(K_i[j], \mathfrak{K})\right)^\frac{j-1}{j}.
\end{equation}
Raising both sides to the $\left(\frac{j}{j-1}\right)$-power establishes the claimed inequality \eqref{eq:induction} for $j$. This completes the induction.
\end{proof}
\begin{proof}[Proof of Theorem~\ref{t:mixed}]
    We first exhaust the cases that do not require iteration. Notice that, when $n=2$, both sides of the inequality are just $\vol_{2m}(\PP(K_1))$, and so there is always equality. When $n=3$, the inequality reads as
    \[
     \vol_{3m}(\PP(K_1,K_{2}))^{2}\leq \vol_{3m}(\PP K_1)\vol_{3m}(\PP K_2),
    \]
    which is just Lemma~\ref{l:mixed_1}. As for the case when $n=4$, we start with Lemma~\ref{l:mixed_2} and obtain
     \[
    \vol_{4m}(\PP(K_1,K_2,K_3))^2 \leq \vol_{4m}(\PP (K_1,K_1,K_3)) \vol_{4m}(\PP (K_2,K_2,K_3) ).
    \]
    We next raise both sides of the above inequality to the $3$rd power and apply Lemma~\ref{l:mixed_1} to each term, yielding
     \begin{align*}
    &\vol_{4m}(\PP(K_1,K_2,K_3))^6 \leq \vol_{4m}(\PP (K_1,K_1,K_3))^3 \vol_{4m}(\PP (K_2,K_2,K_3))^3
    \\
    &\leq \left(\vol_{4m}(\PP K_1)^{2} \vol_{4m}(\PP K_3)\right) \left(\vol_{4m}(\PP K_2)^{2} \vol_{4m}(\PP K_3)\right).
    \end{align*}
    The inequality follows by taking the square root of both sides of the inequality.

    For $n\geq 5$, observe that, if we set $j=n-1$ in Lemma~\ref{l:induction_2}, then the tail $\mathfrak{K}$ is empty, and we have:
\begin{equation*}
\vol_{nm}\left(\PP(K_1, \dots, K_{n-1})\right)^{n-1} \leq \prod_{i=1}^{n-1} \vol_{nm}\left(\PP(K_i)\right),
\end{equation*}
which is the sought-after inequality. But we also want a route that uses Lemma~\ref{l:mixed_1}, to inherit its equality conditions. To obtain this, we set $j=n-2$ in Lemma~\ref{l:induction_2} and deduce, by raising both sides to the $n-1$ power,
\begin{align*}
\vol_{nm}\left(\PP(K_1, \dots, K_{n-2},K_{n-1})\right)^{(n-1)(n-2)} &\leq \prod_{i=1}^{n-2} \vol_{nm}\left(\PP(K_i[n-2], K_{n-1})\right)^{n-1} 
\\
&\leq \prod_{i=1}^{n-1}\vol_{nm}\left(\PP(K_i)\right)^{n-2}.
\end{align*}
Here, we used Lemma~\ref{l:mixed_1} in the final line. Taking the $(n-2)$nd root yields the inequality. For the equality case, we have by Lemma~\ref{l:mixed_1} that the $K_i$ must be homothetic to each other. Notice we also have equality in our use of Lemma~\ref{l:induction_2} in this case.
\end{proof}

{\bf Acknowledgments:} The author was funded by the U.S. National Science Foundation's MSPRF fellowship via NSF grant DMS-2502744. We thank Yao Qiu for a thorough proofreading of a previous draft. We sincerely thank the anonymous referee, whose comments greatly improved the presentation of the facts herein.

\bibliographystyle{acm}
\bibliography{references}

\end{document}